\documentclass[12pt]{amsart}

\usepackage{amsmath, amssymb}

\newtheorem{theorem}{Theorem}[section]

\newtheorem{prop}[theorem]{Proposition}
\newtheorem{cor}[theorem]{Corollary}

\theoremstyle{definition}

\newtheorem{example}[theorem]{Example}

\theoremstyle{remark}
\newtheorem{remark}[theorem]{Remark}

\numberwithin{equation}{section}

\newcommand{\rr}{{\mathbb R}}
\newcommand{\rd}{{\mathbb R^d}}
\newcommand{\nat}{{\mathbb N}}
\newcommand{\ganz}{{\mathbb Z}}
\newcommand{\complex}{{\mathbb C}}
\newcommand{\Exp}{{\mathbb E}}
\newcommand{\Ker}{\operatorname{Ker}}

\newcommand{\fd}{{\,\stackrel{\rm fd}{=}\,}}

\renewcommand{\Re}{{\rm Re}}
\renewcommand{\Im}{{\rm Im}}

\newcommand{\dimp}{\dim_{\mathrm{P}}}

\allowdisplaybreaks

\begin{document}
\sloppy
\title[Asymptotic behavior of semistable L\'evy exponents]{Asymptotic behavior of semistable L\'evy exponents and applications to fractal path properties}

\thanks{The work of P.\ Kern was supported by Deutsche Forschungsgemeinschaft (DFG) under grant KE1741/6-1.  M.M. Meerschaert was partially supported by ARO grant W911NF-15-1-0562 and NSF grants DMS-1462156 and EAR-1344280. Y. Xiao was partially supported by NSF grants DMS-1307470, DMS-1309856 and DMS-1612885}

\author{P.\ Kern}
\address{Peter Kern, Mathematisches Institut, Heinrich-Heine-Universit\"at D\"usseldorf, Universit\"atsstr. 1, D-40225 D\"usseldorf, Germany}
\email{kern\@@{}hhu.de}

\author{M.M.\ Meerschaert}
\address{Mark M.\ Meerschaert, Department of Statistics and Probability, Michigan State University, East Lansing, MI 48824}
\email{mcubed\@@{}stt.msu.edu}

\author{Y.\ Xiao}
\address{Yimin Xiao, Department of Statistics and Probability, Michigan State University, East Lansing, MI 48824}
\email{xiao\@@{}stt.msu.edu}

\date{\today}

\begin{abstract}
This paper proves sharp bounds on the tails of the L\'evy exponent of an operator semistable law on $\rd$. These bounds are then applied to explicitly compute the Hausdorff and packing dimensions of the range, graph, and other random sets describing the sample paths of the corresponding operator semi-selfsimilar L\'evy processes.  The proofs are elementary, using only the properties of the L\'evy exponent, and certain index formulae.
\end{abstract}


\maketitle

\baselineskip=18pt

\section{Introduction}\label{sec1}

Let $X=\{X(t)\}_{t\geq0}$ be a {\it L\'evy process} in $\rd$, i.e.\ a stochastically continuous process with c\`adl\`ag paths that has stationary and independent increments and starts at the origin, i.e.\ $X(0)=0$ almost surely. The distribution of $X$ on the path space is uniquely determined by the distribution of $X(1)$ which can be an arbitrary infinitely divisible distribution in $\rd$. The L\'evy process $X$ is called {\it operator semistable} if the (infinitely divisible) distribution $\mu$ of $X(1)$ is {\it full}, i.e.\ not supported on any lower dimensional hyperplane, and fulfills
\begin{equation}\label{oss}
\mu^c=c^E\mu\ast\varepsilon_u
\end{equation}
for some fixed $c>1$, $u\in\rd$ and some linear operator $E$ on $\rd$, where $\mu^c$ denotes the $c$-fold convolution power of $\mu$, $c^E\mu(dx)=\mu(c^{-E}dx)$ is the image measure of $\mu$ under the exponential operator $c^E=\sum_{n=0}^\infty\frac{(\log c)^n}{n!}\,E^n$, and $\varepsilon_u$ denotes the Dirac measure at the point $u\in\rd$. Operator semistable distributions were introduced by Jajte \cite{Jaj}; further early investigations can be found in \cite{LR,Luc,Cho}. In case $u=0$ the distribution $\mu$, respectively the L\'evy process $X$ generated by $\mu$, is called {\it strictly} operator semistable. Any {\it exponent} $E$ is invertible, and any eigenvalue $\lambda$ of $E$ further fulfills $\Re(\lambda)\geq\frac12$, where $\Re(\lambda)=\frac12$ indicates a Gaussian component \cite[Theorem 7.1.10]{MS}.  We refer to the monograph \cite{MS} for a comprehensive overview on operator semistable distributions and their connection to limit theorems. As an easy consequence of \eqref{oss}, a strictly operator semistable L\'evy process $X$ is also {\it strictly operator semi-selfsimilar}, i.e.,
\begin{equation}\label{soss}
\{X(ct)\}_{t\geq0}\fd\{c^EX(t)\}_{t\geq0},
\end{equation}
where $\fd$ denotes equality of all finite-dimensional marginal distributions of the processes. The class of operator semi-selfsimilar processes is much larger than that of the semi-selfsimilar processes in the literature; see Maejima and Sato \cite{MaeSat} and the references therein for more information. By induction we clearly have $\{X(c^kt)\}_{t\geq0}\fd\{c^{kE}X(t)\}_{t\geq0}$ for all $k\in\ganz$. If \eqref{soss} even holds for all $c>0$ the L\'evy process $X$ is called {\it strictly operator selfsimilar} and the distribution of $X(t)$ is {\it strictly operator stable} \cite[Definition 3.3.24]{MS}.  If $E$ is a scalar multiple of the identity, then an operator (semi-)stable law is called (semi-)stable, and an operator (semi-)selfsimilar process is called (semi-)selfsimilar.  The operator scaling allows the tail behavior to vary with the coordinate, in an arbitrary coordinate system \cite[Theorem 7.1.18]{MS}. This is important for many  applications, including portfolio modeling in finance \cite{portfolio}, pollution plumes in heterogeneous porous media \cite{RW2d}, and diffusion tensor imaging \cite{anisoDTI}.  Hence operator semi-selfsimilarity generalizes the space-time scaling of selfsimilarity to a discrete scale, and allows spatial scaling by linear operators, which gives more flexibility in modeling. We refer to \cite{Sor} for several concrete applications of discrete scale-invariant phenomena from physics.

We remark that operator semistable L\'evy processes are special cases of
\emph{group self-similar processes} introduced by Kolody\'{n}ski
and Rosi\'{n}ski \cite{KR03}. To recall their definition, let $G$ be a group
of transformations of a set $T$ and, for each $(g, t) \in G\times T$,
let $C(g, t): \rd \to \rd$ be a bijection such that
$$C(g_1g_2, t) = C(g_1, g_2(t))\circ C(g_2, t), \quad\text{ for all }g_1, g_2 \in G
\hbox{ and } t \in T,$$
and $C(e, t)= I$. Here $e$ is the unit element of $G$ and $I$ is the identity
operator on $\rd$. In other words, $C$ is a cocycle for the group action
$(g, t)\mapsto g(t)$ of $G$ on $T$. According to Kolody\'{n}ski
and Rosi\'{n}ski \cite{KR03}, a stochastic process $\{Y(t), t \in T\}$
taking values in $\rd$ is called \emph{$G$-self-similar with cocycle $C$} if
\begin{equation}\label{Eq:GSS1}
\{Y\big(g(t)\big),\, t\in T\}\fd \{C(g, t)Y(t),\, t\in T\}.
\end{equation}
In the setting of this paper, we take $T = [0, \infty)$ and
$G = \{g_k: k \in\ganz \}$, where $g_k$ is the transformation on $T$
defined by $g_k (t) = c^k t$. Thus $G$ is a subgroup of linear transformations
on $T$  and we may identify $g_k$ with $c^k$. It is clear that a strictly operator
semi-selfsimilar process $X= \{X(t)\}_{t \ge 0}$ is $G$-self-similar
with cocycle $C$, where for each $g_k \in G$ and $t \ge 0$ the cocykle $C(g_k,t): \rd \to \rd$
is defined by $C(g_k, t)(x) = c^{kE}$. Note that $C(g_k,t)$ is a bijection since $X$
is proper and $c > 1$. Thus, operator semistable L\'evy processes can also be
studied by using the general framework of Kolody\'{n}ski
and Rosi\'{n}ski \cite{KR03} and methods from ergodic theory, but this goes
beyond the scope of the present paper.

We will need the following spectral decomposition of the exponent $E$
as laid out in \cite{MS}. Factor the minimal polynomial of $E$ into
$f_{1}(x)\cdot\ldots\cdot f_{p}(x)$ such that every root of $f_{j}$
has real part $a_{j}$, where $a_1<\cdots<a_p$ are the distinct real parts of the eigenvalues of $E$ and $a_1\geq\frac12$ by 
Theorem 7.1.10 in  \cite{MS}. According to  Theorem 2.1.14 in \cite{MS} we can decompose $\rd$ into a direct sum 
$\rd=V_{1}\oplus\ldots\oplus V_{p}$, where $V_{j}=\Ker(f_{j}(E))$ are $E$-invariant subspaces. Now, in an appropriate basis, 
$E$ can be represented as a block-diagonal matrix $E=E_{1}\oplus\ldots\oplus E_{p}$, where $E_{j}:V_{j}\rightarrow V_{j}$ 
and every eigenvalue of $E_{j}$ has real part $a_{j}$. Especially, every $V_j$ is an $E_j$-invariant subspace of dimension 
$d_j=\dim V_j$. Now we can write $x\in\rd$ as $x=x_1+\cdots+x_p$ and $t^Ex=t^{E_1}x_1+\cdots+t^{E_p}x_p$ with respect 
to this direct sum decomposition, where $x_j\in V_j$ and $t>0$. Moreover, for the strictly operator semistable L\'evy process 
we have $X(t)=X^{(1)}(t)+\ldots+X^{(p)}(t)$ with respect to this direct sum decomposition, where $\{X^{(j)}(t)\}_{t\geq0}$ is an 
operator semistable L\'evy process on $V_j\cong \rr^{d_j}$ with exponent $E_j$ by  Lemma 7.1.17 in \cite{MS}. We can 
further choose an inner product on $\rd$ such that the subspaces $V_j$, $1\leq j\leq p$, are mutually orthogonal and throughout 
this paper for $x\in\rd$ we may choose $\|x\|=\langle x,x\rangle^{1/2}$ as the associated Euclidean norm on $\rd$. With this 
choice, in particular we have
\begin{equation}
	\|X(t)\|^2=\|X^{(1)}(t)\|^{2}+\ldots+\|X^{(p)}(t)\|^{2}
\end{equation}
for all $t\geq0$. Exponents need not be unique, because of possible symmetries \cite[Remark 7.1.22]{MS}. However, since the real parts of the eigenvalues determine the tail behavior of $\mu$ \cite[Theorem 7.1.18]{MS}, the spectral decomposition is the same for any exponent.   In case $d=1$ a spectral decomposition is superfluous and we simply say that $X$ is a strictly {\it $\alpha$-semistable L\'evy process} with $\alpha=\alpha_1=a_1^{-1}=E^{-1}\in(0,2]$.

Since densities of operator semistable distributions exist but are in general not explicitly known, to show properties of operator semistable processes it is often natural to use Fourier transforms which are given in terms of the L\'evy-Khintchine representation.

Our interest is focused on the asymptotic behavior of the {\it L\'evy exponent} $\psi:\rd\to\complex$ which is the unique continuous function with $\psi(0)=0$ and $\Exp[\exp(i\langle\xi,X(t)\rangle)]=\exp(-t\psi(\xi))$ given by the L\'evy-Khintchine formula
$$\psi(\xi)=i\langle\xi,b\rangle+\frac12\langle\xi,\Sigma\xi\rangle+\int_{\rd\setminus\{0\}}\left(1-e^{i\langle\xi,x\rangle}+\frac{i\langle\xi,x\rangle}{1+\|x\|^2}\right)\phi(dx)$$
for some unique $b\in\rd$, a symmetric and non-negative definite $\Sigma\in\rr^{d\times d}$, and a L\'evy measure $\phi$. The latter is a $\sigma$-finite Borel measure on $\rd\setminus\{0\}$ satisfying $$\int_{\rd\setminus\{0\}}\min\{1,\|x\|^2\}\phi(dx)<\infty.$$
Our aim is to derive upper and lower bounds for the real and imaginary part of the L\'evy exponent $\psi$ in terms of the spectral decomposition. These results are presented in Section 2 and enable us to prove upper and lower bounds for $\Re((1+\psi(\xi))^{-1})$ in the operator semistable setup, generalizing the corresponding result for operator stable L\'evy processes given in Proposition 4.1 of \cite{MX}. The quantity $\Re((1+\psi(\xi))^{-1})$ appears in various formulas for the Hausdorff and packing dimensions of certain random sets that describe the sample path behavior of a L\'evy processes. This enables us to give alternative analytic proofs for the Hausdorff and packing dimensions of the range and the graph of operator semistable L\'evy processes in Section 3. We will further show a connection to recurrence properties of operator semistable L\'evy processes and to the Hausdorff dimension of multiple points of their sample paths.

\section{Tail estimates for L\'evy exponents}

Suppose that $X =\{X(t)\}_{t\geq0}$ is operator semistable with exponent $E$.  Recall from Section \ref{sec1} that $a_1<\cdots<a_p$ are the distinct real parts of the eigenvalues of $E$, with $a_1\geq 1/2$, and define $\alpha_i=a_i^{-1}$ so that $2\geq \alpha_1>\cdots>\alpha_p$.  Now we can state the main technical result of this paper. 

\begin{theorem}\label{main}
Let $X=\{X(t)\}_{t\geq0}$ be a strictly operator semistable L\'evy process in $\rd$ with L\'evy exponent $\psi$. Then for every $\varepsilon>0$ there exists $\tau>1$ such that for some $K_i=K_i(\varepsilon,\tau)$ and $\|\xi\|>\tau$ we have
\begin{enumerate}
\item[(a)] $\displaystyle K_2\sum_{i=1}^p\|\xi_i\|^{\alpha_i}\leq\Re(\psi(\xi))\leq K_1\|\xi\|^{\varepsilon/2}\sum_{i=1}^p\|\xi_i\|^{\alpha_i},$
\item[(b)] $\displaystyle|\Im(\psi(\xi))|\leq K_3\|\xi\|^{\varepsilon/2}\sum_{i=1}^p\|\xi_i\|^{\alpha_i}.$
\end{enumerate}
\end{theorem}

\begin{proof}
We will need the following refinement of the spectral decomposition of the exponent $E=E_{1}\oplus\ldots\oplus E_{p}$ with respect to $\rd=V_{1}\oplus\ldots\oplus V_{p}$ laid out in Section \ref{sec1}. Apply the Jordan decomposition to get further direct sums $V_i=U_{i1}\oplus\ldots\oplus U_{iq(i)}$ of $E$-invariant subspaces such that in an appropriate basis $E_i=E_{i1}\oplus\ldots\oplus E_{iq(i)}$ is block-diagonal and every $x\in U_{ij}\setminus\{0\}$ is of order $j$, i.e.\ if we write $E_{ij}=S_{ij}+N_{ij}$, where $S_{ij}$ is semisimple and $N_{ij}$ is nilpotent, then $N_{ij}^{j-1}x\not=0$ and $N_{ij}^{j}x=0$.  This $S+N$ decomposition is unique, e.g., see Hoffman and Kunze \cite{HK61}.  If we write $x=\sum_{i=1}^p\sum_{j=1}^{q(i)}x_{ij}$ with respect to these direct sum decompositions, so that $x_{ij}\in U_{ij}$, by the proof of Theorem 2.2.4 in \cite{MS} we have in an associated Euclidean norm
$$\|t^{-E^\ast}x\|^2=\sum_{i=1}^p\sum_{j=1}^{q(i)}\frac{t^{-2a_i}(\log t)^{2(j-1)}}{((j-1)!)^2}\,\|x_{ij}\|^2+o_{ij}(t,x),$$
where $E^\ast$ denotes the adjoint of the exponent $E$ and $o_{ij}(t,x)$ is a linear combination of terms of the form $t^{-2a_i}(\log t)^m$ with $m<2(j-1)$. Then for fixed $x\not=0$ the function $t\mapsto R(t)=\|t^{E^\ast}x\|^{-1}$ is regularly varying with index $a=\min\{a_i:\,x_i\not=0\}$. Now let $r\mapsto t(r)$ be an asymptotic inverse of $R(t)$, i.e.\ a regularly varying function with index $\alpha=a^{-1}$ such that $R(t(r))/r\to1$ as $r\to\infty$. An explicit choice of
\begin{equation}\label{asyinv}
t(r)=\sum_{i=1}^p\sum_{j=1}^{q(i)}\left(\frac{\alpha_i^{j-1}}{(j-1)!}\right)^{\alpha_i}r^{\alpha_i}(\log r)^{\alpha_i(j-1)}\|x_{ij}\|^{\alpha_i}
\end{equation}
shows that the convergence $R(t(r))/r\to1$ holds uniformly on compact subsets of $\{x\not=0\}$. For a more detailed derivation of \eqref{asyinv} we refer to the proof of Theorem 4.2 in \cite{MX}.

Write $t>0$ as $t=c^{k(t)}m(t)$ with $k(t)\in\ganz$ and $m(t)\in[1,c)$. By \eqref{soss} we have that $X(t)$ and $c^{k(t)E}X(m(t))$ are identically distributed and hence
\begin{equation}\label{psioss}
t\,\psi(\xi)=m(t)\,\psi(c^{k(t)E^\ast}\xi)\quad\text{ for all }t>0,\,\xi\in\rd.
\end{equation}
Let $F(\xi)=\Re(\psi(\xi))$, then by \eqref{psioss} we get
\begin{equation}\label{Repsioss}
t\,F(\xi)=m(t)\,F(c^{k(t)E^\ast}\xi)\quad\text{ for all }t>0,\,\xi\in\rd.
\end{equation}
Since $X$ is full, $F(\xi)$ is bounded away from zero and infinity on compact subsets of $\{\xi\not=0\}$ by Corollary 7.1.12 in \cite{MS}. Given $x\not=0$ and $r>0$ define $\theta_{r,x}=t(r)^{-E^\ast}rx$ using \eqref{asyinv}. Then
$\|\theta_{r,x}\|=r\|t(r)^{-E^\ast}x\|=r/R(t(r))\to1$ as $r\to\infty$ uniformly on compact subsets of $\{x\not=0\}$. Hence, given $\eta\in(0,1)$ there exists $r_0>0$ such that
\begin{equation}\label{thetab}
1-\eta<\|\theta_{r,x}\|<1+\eta\quad\text{ for all }r\geq r_0,\,x\in S_d.
\end{equation}
For $\xi\not=0$ let $r=\|\xi\|>0$ and $x=\xi/r\in S_d$, then by \eqref{Repsioss} we have
\begin{equation}\label{Foss}\begin{split}
F(\xi) & =F(rx)=F(t(r)^{E^\ast}\theta_{r,x})=F(c^{k(t(r))E^\ast}m(t(r))^{E^\ast}\theta_{r,x})\\
& = m(t(r))^{-1}t(r)F(m(t(r))^{E^\ast}\theta_{r,x})=c^{k(t(r))}F(m(t(r))^{E^\ast}\theta_{r,x})
\end{split}\end{equation}
and, since $m(t(r))\in[1,c)$ together with \eqref{thetab} we get that $F(m(t(r))^{E^\ast}\theta_{r,x})$ is uniformly bounded away from zero and infinity for all $r\geq r_0$ and $x\in S_d$.

Now let $\varepsilon>0$ be given and choose a constant $\tau\geq\max\{r_0,e\}$ such that for all $r\geq\tau$ we have $1\leq(\log r)^{\alpha_i(q(i)-1)}\leq r^{\varepsilon/2}$ for all $1\leq i\leq p$. Then it follows from \eqref{Foss} and \eqref{asyinv} that for all $r\geq\tau$ we have
\begin{equation}\label{Fossu}\begin{split}
F(\xi) & =c^{k(t(r))}F(m(t(r))^{E^\ast}\theta_{r,x})\leq K_1'c^{k(t(r))}m(t(r))\\
& =K_1't(r)=\tilde K_1\sum_{i=1}^p\sum_{j=1}^{q(i)}r^{\alpha_i}(\log r)^{\alpha_i(j-1)}\|x_{ij}\|^{\alpha_i}\\
& \leq\tilde K_1 r^{\varepsilon/2}\sum_{i=1}^pr^{\alpha_i}\sum_{j=1}^{q(i)}\|x_{ij}\|^{\alpha_i}\\
& \leq K_1\|\xi\|^{\varepsilon/2}\sum_{i=1}^p(r\|x_i\|)^{\alpha_i}=K_1\|\xi\|^{\varepsilon/2}\sum_{i=1}^p\|\xi_i\|^{\alpha_i},
\end{split}\end{equation}
where the constant $K_1$ does not depend on $\xi$ and the inequality in the last line follows from
$\sum_{j=1}^{q(i)}\|x_{ij}\|^{\alpha_i}\leq q(i)\|x_i\|^{\alpha_i}\leq d\|x_i\|^{\alpha_i}$. This proves part (a).

Similarly, for part (b) it follows from \eqref{Foss} and \eqref{asyinv} that for all $r\geq\tau$ we have
\begin{equation}\label{Fossl}\begin{split}
F(\xi) & =c^{k(t(r))}F(m(t(r))^{E^\ast}\theta_{r,x})\geq K_2' c^{k(t(r))}m(t(r))\\
& =K_2't(r)=\tilde K_2\sum_{i=1}^p\sum_{j=1}^{q(i)}r^{\alpha_i}(\log r)^{\alpha_i(j-1)}\|x_{ij}\|^{\alpha_i}\\
& \geq\tilde K_2\sum_{i=1}^pr^{\alpha_i}\sum_{j=1}^{q(i)}\|x_{ij}\|^{\alpha_i}\\
& \geq K_2\sum_{i=1}^p(r\|x_i\|)^{\alpha_i}=K_2\sum_{i=1}^p\|\xi_i\|^{\alpha_i},
\end{split}\end{equation}
where the constant $K_2$ does not depend on $\xi$ and the inequality in the last line follows from
$\|x_i\|^{\alpha_i}=\|\sum_{j=1}^{q(i)}x_{ij}\|^{\alpha_i}\leq C_1(\sum_{j=1}^{q(i)}\|x_{ij}\|^2)^{\alpha_i/2}\leq C_1\sum_{j=1}^{q(i)}\|x_{ij}\|^{\alpha_i}$.

Now we consider $G(\xi)=\Im(\psi(\xi))$ for which by \eqref{psioss} we have
\begin{equation*}
t\cdot G(\xi)=m(t)\cdot G(c^{k(t)E^\ast}\xi)\quad\text{ for all }t>0,\,\xi\in\rd
\end{equation*}
and $G$ is bounded on compact subsets of $\rd\setminus\{0\}$ by continuity of $\psi$. Hence as above we get for all $\|\xi\|=r\geq\tau$
\begin{equation}\label{Gossu}
|G(\xi)|=c^{k(t(r))}|G(m(t(r))^{E^\ast}\theta_{r,x})|\leq K_3't(r)\leq K_3\|\xi\|^{\varepsilon/2}\sum_{i=1}^p\|\xi_i\|^{\alpha_i},
\end{equation}
where the constant $K_3$ does not depend on $\xi$, proving part (c).
\end{proof}

\begin{cor}\label{proppsi}
Let $X$ be a strictly operator semistable L\'evy process in $\rd$ with L\'evy exponent $\psi$. Then for every $\varepsilon>0$ there exists $\tau>1$ such that for some $K=K(\varepsilon,\tau)$ we have
\begin{equation}\label{psibou}
\frac{K^{-1}\|\xi\|^{-\varepsilon}}{\sum_{i=1}^p\|\xi_i\|^{\alpha_i}}
\leq\Re\left(\frac1{1+\psi(\xi)}\right)\leq\frac{K}
{\sum_{i=1}^p\|\xi_i\|^{\alpha_i}}\quad\text{ for all }\|\xi\|>\tau.
\end{equation}
\end{cor}

\begin{proof}
Using the obvious identity
$$\Re\left(\frac1{1+\psi(\xi)}\right)=\frac{1+\Re(\psi(\xi))}{(1+\Re(\psi(\xi)))^2+(\Im(\psi(\xi))^2}=\frac{1+F(\xi)}{(1+F(\xi))^2+(G(\xi))^2},$$
by Theorem \ref{main} we get for all $\xi\in\rd$ with $\|\xi\|\geq\tau$
$$\Re\left(\frac1{1+\psi(\xi)}\right)\leq\frac1{1+F(\xi)}\leq\frac1{F(\xi)}\leq\frac{K_2^{-1}}{\sum_{i=1}^p\|\xi_i\|^{\alpha_i}}$$
and
\begin{align*}
\Re\left(\frac1{1+\psi(\xi)}\right) & \geq\frac{F(\xi)}{(1+F(\xi)))^2+(G(\xi))^2}\\
& \geq\frac{K_2\sum_{i=1}^p\|\xi_i\|^{\alpha_i}}{\left(1+K_1\|\xi\|^{\varepsilon/2}\sum_{i=1}^p\|\xi_i\|^{\alpha_i}\right)^2+\left(K_3\|\xi\|^{\varepsilon/2}\sum_{i=1}^p\|\xi_i\|^{\alpha_i}\right)^2}\\
& \geq K_{12}\,\frac{\sum_{i=1}^p\|\xi_i\|^{\alpha_i}}{\left(\|\xi\|^{\varepsilon/2}\sum_{i=1}^p\|\xi_i\|^{\alpha_i}\right)^2 }=\frac{K_{12}\|\xi\|^{-\varepsilon}}{\sum_{i=1}^p\|\xi_i\|^{\alpha_i}},
\end{align*}
concluding the proof.
\end{proof}

\section{Applications to fractal path properties}

\subsection{Range and Graph}

We will now apply the results of Section 2 to derive fractal properties of the range $X([0,1])=\{X(t):\,t\in [0,1]\}$ 
and the graph $G_X([0,1])=\{(t,X(t)):\,t\in [0,1]\}$ of a strictly operator semistable L\'evy process $X$ in terms 
of their Hausdorff and packing dimensions.
We refer to \cite{Fal90} for a systemic account on fractal dimensions and their properties.
With the help of the spectral decomposition of the exponent $E$ the Hausdorff dimension of the range of a 
strictly operator semistable L\'evy process in $\rd$ with $d\geq2$ has been calculated in Theorem 3.1 of \cite{KW1} as
\begin{equation}\label{dimr2}
\dim_{\rm H}X(B)=\begin{cases}\alpha_1\dim_{\rm H}B & \text{ if }\alpha_1\dim_{\rm H}B\leq d_1\\ 1+
\alpha_2(\dim_{\rm H}B-\alpha_1^{-1}) & \text{ else}\end{cases}
\end{equation}
almost surely, where $B\in\mathcal B(\rr_+)$ is an arbitrary Borel set. In case $d=1$ by Theorem 3.3 in \cite{KW1} for a strictly $\alpha$-semistable L\'evy process we have
\begin{equation}\label{dimr1}
\dim_{\rm H}X(B)=\min\{\alpha\dim_{\rm H}B,1\}
\end{equation}
almost surely. In the special case of a strictly operator stable L\'evy process the formula \eqref{dimr2} was established by Meerschaert and Xiao \cite{MX} generalizing an earlier partial result for $B=[0,1]$ in \cite{BMS}. Moreover, the Hausdorff dimension of the graph of a strictly operator semistable L\'evy process in $\rd$ with $d\geq2$ was recently calculated in Theorem 3.1 of \cite{Wed} as
\begin{equation}\label{dimg2}
\dim_{\rm H}G_X(B)=\begin{cases}\dim_{\rm H}B\cdot\max\{\alpha_1,1\} & \text{ if }\alpha_1\dim_{\rm H}B\leq d_1\\ 1+\max\{\alpha_2,1\}(\dim_{\rm H}B-\alpha_1^{-1}) & \text{ else}\end{cases}
\end{equation}
almost surely, and in case $d=1$ by Theorem 3.2 in \cite{Wed} for a strictly $\alpha$-semistable L\'evy process we have
\begin{equation}\label{dimg1}
\dim_{\rm H}G_X(B)=\begin{cases}\dim_{\rm H}B\cdot\max\{\alpha,1\} & \text{ if }\alpha\dim_{\rm H}B\leq 1\\ 1+\dim_{\rm H}B-\alpha^{-1} & \text{ else}\end{cases}
\end{equation}
almost surely. The derivation of \eqref{dimr2}--\eqref{dimg1} in \cite{KW1,Wed} uses the standard method of showing that 
almost surely the right-hand side in \eqref{dimr2}--\eqref{dimg1} serves as an upper as well as a lower bound for the 
Hausdorff dimension on the corresponding left-hand side, following classical results for the range of one-dimensional 
stable L\'evy processes in Blumenthal and Getoor \cite{BG,BG1,BG2} and Hendricks \cite{Hen}, and L\'evy processes 
with independent stable components in Pruitt and Taylor \cite{Tay,Pru,PT}. The lower bound is shown by an application 
of Frostman's capacity theorem to prove that certain expected energy integrals are finite. The upper bound is shown 
by using the covering lemma of Pruitt and Taylor \cite[Lemma 6.1]{PT} which needs sharp lower bounds for the expected 
sojourn time in small balls. For the latter in \cite{KW1} uniform density bounds were derived in the semistable situation. 
For an overview we refer to the survey article \cite{Xiao}.

An alternative analytic approach for $B=[0,1]$ uses an index formula proved in Corollary 1.8 of \cite{KXZ}, 
valid for arbitrary L\'evy processes $X$ in $\rd$, which states that almost surely
\begin{equation}\label{Hdimcrit}
\dim_{\rm H}X([0,1])=\sup\left\{a<d:\,\int_{\{\|\xi\|\geq1\}}\Re\left(\frac1{1+\psi(\xi)}\right)\,
\frac{d\xi}{\|\xi\|^{d-a}}<\infty\right\}.
\end{equation}
Similarly, Khoshnevisan and Xiao \cite{KX08} established
the following formula for the packing dimension of $X([0,1])$ in terms of the L\'evy exponent 
$\psi(\xi)$:
\begin{equation}\label{Eq:KX08}
 \dimp X([0\,,1]) = \sup\left\{\eta\ge 0:\ \liminf_{r\to 0^+}
        \frac{W(r)}{r^\eta}  =0\right\}
        = \limsup_{r\to 0^+} \frac{\log W(r)} {\log r},
\end{equation}
almost surely, where $\sup\varnothing :=0$ and the function $W$ is defined by
\begin{equation}\label{def:W}
    W(r)= \int_{\rd}  \Re\left( \frac{1}{1+\psi(\frac x r)}
    \right)\frac1{ \prod_{j=1}^d
    (1+ x^2_j)}\, dx.
\end{equation}
In \cite[Eq. (1.4)]{KX08} they also provided a formula for
$\dim_{\rm H}X([0,1])$ in terms of $W$.
Notice that, when applied to the L\'evy process $\{(t, X(t)): t \ge 0\}$,
(\ref{Hdimcrit}) and (\ref{Eq:KX08}) also provide analytic ways for
computing the Hausdorff and packing dimensions of the graph of $X$.

Meerschaert and Xiao \cite[Proposition 4.1]{MX} used \eqref{Hdimcrit} to give an alternative proof for \eqref{dimr2} in case $X$ is a full, strictly operator stable L\'evy process and $B=[0,1]$ using bounds for $\Re((1+\psi(\xi))^{-1})$ as in \eqref{psibou}. See also in Proposition 7.7 of \cite{KX}. Khoshnevisan and Xiao \cite[Theorem 3.1]{KX08} showed that, under condition \eqref{psibou}, the packing dimension of $X([0,1])$ is also given by the right-hand side of \eqref{dimr2} with $B=[0,1]$.  Using Corollary \ref{proppsi}, we immediately obtain the following special case of \eqref{dimg2}.

\begin{theorem}\label{MMMadded}
Let $X$ be a strictly operator semistable L\'evy process in $\rd$ with $d\geq2$. Then \begin{equation}\label{dimrange}
\dim_{\rm H}X([0,1])=\dimp X([0\,,1]) = \begin{cases}\alpha_1 & \text{ if }\alpha_1\leq d_1\\ 1+\alpha_2(1-\alpha_1^{-1}) & \text{ else}\end{cases}
\end{equation}
almost surely, in accordance with \eqref{dimr2}. 
\end{theorem}

\begin{proof}
Use Corollary \ref{proppsi} and follow the arguments for \cite[Proposition 4.1]{MX} and \cite[Theorem 3.1]{KX08}.
\end{proof}

We can also obtain a special case of \eqref{dimr1} as follows.

\begin{cor}
Let $X$ be a strictly $\alpha$-semistable L\'evy process in $\rr$. Then 
$$\dim_{\rm H}X([0,1])= \dimp X([0\,,1]) =\min\{\alpha,1\}$$
almost surely.
\end{cor}

\begin{proof}
In case $d=1$ the conclusion \eqref{psibou} of Corollary \ref{proppsi} reads as
\begin{equation}\label{psibou1}
K^{-1}|\xi|^{-\varepsilon-\alpha}\leq\Re\left(\frac1{1+\psi(\xi)}\right)\leq K|\xi|^{-\alpha}\quad\text{ for all }|\xi|>\tau.
\end{equation}
Note that for $d=1$ we can strengthen \eqref{psibou1} to
\begin{equation}\label{psibou1s}
K^{-1}|\xi|^{-\alpha}\leq\Re\left(\frac1{1+\psi(\xi)}\right)\leq K |\xi|^{-\alpha}\quad\text{ for all }|\xi|\geq1,
\end{equation}
since in this case $R(t)=t^{1/\alpha}|x|^{-1}$ and the asymptotic inverse can be chosen as $t(r)=(r|x|)^\alpha$ such that $R(t(r))=r$ for all $r>0$. Following the line of arguments given in the proof of Theorem \ref{main}, it is easy to arrive at \eqref{psibou1s} instead of \eqref{psibou1}. Using \eqref{psibou1s} it is obvious that
$$\int_{|\xi|\geq1}\Re\left(\frac1{1+\psi(\xi)}\right)\,\frac{d\xi}{|\xi|^{1-a}}<\infty\quad\iff\quad a<\alpha.$$
Hence by \eqref{Hdimcrit} we immediately get
$$\dim_{\rm H}X([0,1])=\sup\left\{a<1:\,\int_{\{|\xi|\geq1\}}
\Re\left(\frac1{1+\psi(\xi)}\right)\,\frac{d\xi}{|\xi|^{1-a}}<\infty\right\}=\min\{\alpha,1\}$$
almost surely. Since $\dim_{\rm H}X([0,1]) \le \dimp X([0\,,1]) \le 1$, we see that, in order to prove $\dimp X([0\,,1]) =\min\{\alpha,1\}$ a.s., it is sufficient to consider the case $\alpha < 1$ and verify $\dimp X([0\,,1]) \le \alpha$ a.s. It follows from (\ref{def:W}) and (\ref{psibou1s}) that for $r \in (0, 1)$,
\begin{equation}\label{Eq:case1a}
W(r) \ge K^{-1}\,  r^{\alpha} \int_{\rr}
            \frac{1}{|x|^{\alpha}
            ( 1+ x^2)} = K\, r^{\alpha},
\end{equation}
which implies that $\lim_{r \to 0} r^{-\eta} {W}(r) = \infty$ for all
$\eta > \alpha$. By (\ref{Eq:KX08}), we obtain
$\dimp X([0\,,1])$ $\le \alpha$ a.s. This concludes the proof.
\end{proof}

We now turn to the graph process $\{(t,X(t))\}_{t\geq0}$ which is a L\'evy process in $\rr^{d+1}$ such that \eqref{oss} holds with block diagonal exponent $1\oplus E$. Since the L\'evy measure of the graph process is concentrated on a $d$-dimensional subspace of $\rr^{d+1}$, it is not full \cite[Proposition 3.1.20]{MS}, and hence it is not operator semistable.  However, neither \eqref{Hdimcrit} nor \eqref{Eq:KX08} assume fullness of the L\'evy process.  

Write $\tilde\xi=(\xi_0,\xi)\in\rr^{d+1}$ with $\xi=(\xi_1,\ldots,\xi_d)\in\rd$ and observe that the L\'evy exponent $\tilde\psi$ of the graph process is given by $\tilde\psi(\tilde\xi)=\psi(\xi)-i\xi_0$, which leads to
\begin{equation}\label{intnew}
\Re\left(\frac1{1+\tilde\psi(\tilde\xi)}\right)=\frac{1+F(\xi)}{(1+F(\xi))^2+(G(\xi)-\xi_0)^2}=:H(\xi_0,\xi),
\end{equation}
where $F=\Re\psi$ and $G=\Im\psi$ are as in the proof of Theorem \ref{main}.  Next we prove a special case of \eqref{dimg2}.  The proof is elementary, using only the sharp bounds of Theorem \ref{main} along with the index formulae \eqref{Hdimcrit} and \eqref{Eq:KX08}.

\begin{theorem}\label{dimgraph}
Let $X$ be a strictly operator semistable L\'evy process in $\rd$ with $d\geq2$. Then 
$$\dim_{\rm H}G_X([0,1])= \dim_{\rm P}G_X([0,1])=\begin{cases}\max\{\alpha_1,1\} & \text{ if }\alpha_1\leq d_1\\ 1+\max\{\alpha_2,1\}(1-\alpha_1^{-1}) & \text{ else}\end{cases}$$
almost surely.
\end{theorem}

To clarify the proof of Theorem \ref{dimgraph}, it will be helpful to derive the corresponding statement for $d=1$ first.  This one dimensional result is a special case of \eqref{dimg1}.

\begin{prop}\label{dimgraph1}
Let $X$ be a strictly $\alpha$-semistable L\'evy process in $\rr$. Then 
\[
\begin{split}
\dim_{\rm H}G_X([0,1])&= \dim_{\rm P}G_X([0,1]) = \max\{1,2-\alpha^{-1}\} =\begin{cases}1 & \text{ if }\alpha\leq 1\\ 2-\alpha^{-1} & \text{ else}
\end{cases}
\end{split}
\]
almost surely.
\end{prop}

In the next two proofs, $K$ denotes an unspecified positive constant whose value may vary at each occurrence.

\begin{proof}
We will first establish lower bounds. In case $\alpha\leq1$ clearly $\dim_{\rm H}G_X([0\,,1])\geq1$ by projecting the graph $\{(t,X(t))\}_{t\geq0}$ onto the first (deterministic) component. In case $\alpha>1$ let $\gamma\in(0,2-\alpha^{-1})$ and then note that in view of \eqref{Hdimcrit} and \eqref{intnew} we need to show that
\begin{align*}
I_\gamma:= & \int_{\{|\xi_0|\geq2K''_3,|\xi|\geq\tau\}}\Re\left(\frac1{1+\tilde\psi(\tilde\xi)}\right)\,\frac{d\tilde\xi}{\|\tilde\xi\|^{2-\gamma}}\\
= & \left(\int_{A_1}+\int_{A_2}+\int_{A_3}\right)\frac{H(\xi_0,\xi)}{(\xi_0^2+\xi^2)^{1-\gamma/2}}\,d\tilde\xi<\infty,
\end{align*}
where we use a similar decomposition of the domain of integration as in the proof of Theorem 2.1 in Manstavi\v{c}ius \cite{Man}; cf.\ also Figure 1 in \cite{Man}. Namely we set
\begin{align*}
A_1 & =\left\{|\xi_0|\geq2K''_3\max\left\{1,\left(|\xi|/\tau\right)^q\right\}\right\},\\
A_2 & =\left\{|\xi|\geq\tau, |\xi_0|\leq2K''_3|\xi|/\tau\right\},\\
A_3 & =\left\{|\xi|\geq\tau, 2K''_3|\xi|/\tau<|\xi_0|<2K''_3\left(|\xi|/\tau\right)^q\right\},
\end{align*}
where $q=\alpha+\varepsilon/2>1$, $K''_3=K_3\tau^q$ and $\tau>1$ is chosen such that Theorem \ref{main} holds for $\varepsilon>0$ with the following constraints. Since we always have $2-\alpha^{-1}\leq\alpha$, we know that $\gamma<\alpha$ and can choose $\varepsilon>0$ such that $\gamma<\frac{2\alpha-1+\varepsilon/2}{\alpha+\varepsilon/2}$.

On $A_1$ we use $(\xi_0^2+\xi^2)^{1-\gamma/2}\geq|\xi_0|^{2-\gamma}$ and by Theorem \ref{main}(b) we have
$$(G(\xi)-\xi_0)^2\geq\left(|\xi_0|-|G(\xi)|\right)^2\geq\left(|\xi_0|-K_3|\xi|^q\right)^2\geq\left(|\xi_0|/2\right)^2$$
so that, using also part (a) of Theorem \ref{main}, we get by \eqref{intnew}
$$H(\xi_0,\xi)\leq4\,\frac{1+F(\xi)}{\xi_0^2}\leq K\,\frac{|\xi|^{\alpha+\varepsilon/2}}{\xi_0^2}.$$
Hence, using symmetry with respect to $\xi_0$, we get
\begin{align*}
\int_{A_1}\frac{H(\xi_0,\xi)}{(\xi_0^2+\xi^2)^{1-\gamma/2}}\,d\tilde\xi & \leq K\int_{A_1}\frac{|\xi|^{\alpha+\varepsilon/2}}{|\xi_0|^{4-\gamma}}\,d\tilde\xi\\
& = K\int_{2K''_3}^\infty\frac1{\xi_0^{4-\gamma}}\int_{\{|\xi|^q\leq\tau^q/(2K''_3)\,\xi_0\}}|\xi|^{\alpha+\varepsilon/2}\,d\xi\,d\xi_0\\
& \leq K\int_{2K''_3}^\infty\frac1{\xi_0^{4-\gamma}}\int_{\{|\xi|\leq\tau/(2K''_3)^{1/q}\,\xi_0^{1/q}\}}|\xi_0|\,d\xi\,d\xi_0\\
& \leq K\int_{2K''_3}^\infty\frac{\xi_0^{1+\frac1{q}}}{\xi_0^{4-\gamma}}\,d\xi_0=K\int_{2K''_3}^\infty\frac{1}{\xi_0^{3-\gamma-\frac1{q}}}\,d\xi_0<\infty,
\end{align*}
since $\gamma<2-\alpha^{-1}<2-(\alpha+\varepsilon/2)^{-1}=2-q^{-1}$.

On $A_2$ we use $(\xi_0^2+\xi^2)^{1-\gamma/2}\geq|\xi|^{2-\gamma}$ and by Theorem \ref{main}(a) we get
$$H(\xi_0,\xi)\leq\frac1{1+F(\xi)}\leq K\,\frac1{|\xi|^{\alpha}}.$$
Hence, using symmetry with respect to $\xi$, we get
\begin{align*}
\int_{A_2}\frac{H(\xi_0,\xi)}{(\xi_0^2+\xi^2)^{1-\gamma/2}}\,d\tilde\xi & \leq K\int_{A_2}\frac1{|\xi|^{\alpha}}\,\frac1{|\xi|^{2-\gamma}}\,d\tilde\xi\\
& = K\int_{\tau}^\infty\frac1{\xi^{2-\gamma+\alpha}}\int_{\{|\xi_0|\leq2K''_3|\xi|/\tau\}}d\xi_0\,d\xi\\
& =K\int_{\tau}^\infty\frac{1}{\xi^{1-\gamma+\alpha}}\,d\xi<\infty,
\end{align*}
since $\gamma<\alpha$.

On $A_3$ we use $(\xi_0^2+\xi^2)^{1-\gamma/2}\geq|\xi_0|^{2-\gamma}$ as on $A_1$ and by Theorem \ref{main}(a) we get
$$H(\xi_0,\xi)\leq\frac1{1+F(\xi)}\leq K\,\frac1{|\xi|^{\alpha}}$$
as on $A_2$. Hence, using symmetry with respect to $\xi_0$ and $\xi$, we get
\begin{align*}
\int_{A_3}\frac{H(\xi_0,\xi)}{(\xi_0^2+\xi^2)^{1-\gamma/2}}\,d\tilde\xi & \leq K\int_{A_3}\frac1{|\xi|^{\alpha}}\,\frac1{|\xi_0|^{2-\gamma}}\,d\tilde\xi\\
& = K\int_{\tau}^\infty\frac1{\xi^{\alpha}}\int_{2K''_3|\xi|/\tau}^{2K''_3(|\xi|/\tau)^q}\frac1{|\xi_0|^{2-\gamma}}\,d\xi_0\,d\xi\\
& \leq K\int_{\tau}^\infty\frac{\xi^{q(\gamma-1)}}{\xi^{\alpha}}\,d\xi=K\int_{\tau}^\infty\frac{1}{\xi^{2\alpha+\varepsilon/2-\gamma(\alpha+\varepsilon/2)}}\,d\xi<\infty,
\end{align*}
since $\gamma<\frac{2\alpha-1+\varepsilon/2}{\alpha+\varepsilon/2}$ by our choice of $\varepsilon>0$.

Altogether we have shown that $I_\gamma<\infty$ for every $0<\gamma<2-\alpha^{-1}$ so that $\dim_{\rm H}G_X([0\,,1])\geq 2-\alpha^{-1}$ almost surely for $\alpha>1$.

Since $\dim_{\rm H}G_X([0\,,1])\leq\dim_{\rm P}G_X([0\,,1])$ it remains to prove the upper bound for the packing dimension.  In the following we obtain the upper bound in a similar manner as in the proof of Theorem 3.1 in \cite{KX08}. In case $\alpha\in(0,1)$ let $\eta>1$ be arbitrary and choose $\varepsilon>0$ such that $\alpha-1+\varepsilon<0$ and $\eta>1+\varepsilon$. Recall that $\tau>1$ in Theorem \ref{main}.  Then it follows from \eqref{def:W}, \eqref{intnew} and Theorem \ref{main} that for $r\in(0,1)$ and hence $r<r^{1-\alpha-\varepsilon}$ we have
\begin{align*}
W(r) & =\int_{\rr^2} \Re\left( \frac{1}{1+\tilde\psi(\tilde\xi/r)}\right)\frac{d\tilde\xi}{(1+ \xi_0^2)(1+\xi^2)}\\
& \geq K\int_{\tau r}^\infty\int_{\tau r}^\infty\frac{(\xi/r)^{\alpha}}{(\xi/r)^{2\alpha+\varepsilon}+\left((\xi/r)^{\alpha+\varepsilon/2}+\xi_0/r\right)^2}\,\frac{d\xi}{1+\xi^2}\,\frac{d\xi_0}{1+\xi_0^2}\\
& \geq K r^{\alpha+\varepsilon}\int_{\tau r}^\infty\int_{\tau }^\infty\frac{\xi^{\alpha}}{\xi^{2\alpha+\varepsilon}+\left(\xi^{\alpha+\varepsilon/2}+r^{\alpha-1+\varepsilon/2}\xi_0\right)^2}\,\frac{d\xi}{1+\xi^2}\,\frac{d\xi_0}{1+\xi_0^2}\\
& \geq K r^{\alpha+\varepsilon}\int_{\tau r}^\infty\,\frac1{1+(1+r^{\alpha-1+\varepsilon/2}\xi_0)^2}\int_{\tau }^\infty\frac{d\xi}{\xi^{\alpha+\varepsilon}(1+\xi)^2}\,\frac{d\xi_0}{1+\xi_0^2}\\
& \geq K r^{\alpha+\varepsilon}\int_{\tau r^{1-\alpha-\varepsilon}}^1\frac{d\xi_0}{r^{2\alpha-2+\varepsilon}\xi_0^2(1+\xi_0^2)}\\
& \geq  K r^{2-\alpha}\int_{\tau r^{1-\alpha-\varepsilon}}^1\frac{d\xi_0}{\xi_0^2}= Kr^{1+\varepsilon}.
\end{align*}
This implies $\lim_{r \to 0^+} r^{-\eta} {W}(r) = \infty$, since $\eta>1+\varepsilon$ by our choice of $\varepsilon>0$. Hence, by \eqref{Eq:KX08} we obtain $\dimp G_X([0\,,1])\leq1$ almost surely, since $\eta>1$ is arbitrary.

In case $\alpha\geq1$ let $\eta > 2-\alpha^{-1}$ be arbitrary and choose $\varepsilon>0$ such that $\alpha>1+\varepsilon/2$ and $\eta > 2-\alpha^{-1}+2\varepsilon$. Note that
$$\beta:=\frac{2-\alpha-\alpha^{-1}+\varepsilon}{1-\alpha-\varepsilon}=1+\frac{1-\alpha^{-1}+2\varepsilon}{1-\alpha-\varepsilon}<1$$
and observe that $1-(\alpha+\varepsilon/2)(1-\beta)<0$ by elementary calculations.
Then it follows from \eqref{def:W}, \eqref{intnew} and Theorem \ref{main} that for $r\in(0,1)$ and hence $r<r^{\beta}$ we have
\begin{align*}
W(r) & =\int_{\rr^2} \Re\left( \frac{1}{1+\tilde\psi(\tilde\xi/r)}\right)\frac{d\tilde\xi}{(1+ \xi_0^2)(1+\xi^2)}\\
& \geq K\int_{\tau r}^\infty\int_{\tau r}^\infty\frac{(\xi/r)^{\alpha}}{(\xi/r)^{2\alpha+\varepsilon}+\left((\xi/r)^{\alpha+\varepsilon/2}+\xi_0/r\right)^2}\,\frac{d\xi_0}{1+\xi_0^2}\,\frac{d\xi}{1+\xi^2}\\
& \geq K r^{2-\alpha}\int_{\tau r}^\infty\int_{\tau }^\infty\frac{\xi^{\alpha}}{\xi^{2\alpha+\varepsilon}r^{2-2\alpha-\varepsilon}+\left(\xi^{\alpha+\varepsilon/2}r^{1-\alpha-\varepsilon/2}+\xi_0\right)^2}\,\frac{d\xi_0}{1+\xi_0^2}\,\frac{d\xi}{1+\xi^2}\\
& \geq K r^{2-\alpha}\int_{\tau r^\beta}^1\,\frac{\xi^{\alpha}}{\xi^{2\alpha+\varepsilon}r^{2-2\alpha-\varepsilon}+\left(\xi^{\alpha+\varepsilon/2}r^{1-\alpha-\varepsilon/2}+1\right)^2}\int_{\tau }^\infty\frac{d\xi_0}{\xi_0^2(1+\xi_0)^2}\,\frac{d\xi}{1+\xi^2}.
\end{align*}
Since for $\xi\geq\tau r^\beta$ we have $\xi^{\alpha+\varepsilon/2}r^{1-\alpha-\varepsilon/2}\geq K\,r^{1-(\alpha+\varepsilon/2)(1-\beta)}\to\infty$ as $r\to0^+$ by our choice of $\beta$, for sufficiently small $r\in(0,1)$ we get
\begin{align*}
W(r) & \geq K r^{2-\alpha}\int_{\tau r^\beta}^1\,\frac{\xi^{\alpha}}{\xi^{2\alpha+\varepsilon}r^{2-2\alpha-\varepsilon}}\,\frac{d\xi}{1+\xi^2}\\
& \geq K r^{\alpha+\varepsilon}\int_{\tau r^{\beta}}^1\frac{d\xi}{\xi^{\alpha+\varepsilon}}
\geq  K r^{\alpha+\varepsilon+\beta(1-\alpha-\varepsilon)}= Kr^{2-\alpha^{-1}+2\varepsilon}.
\end{align*}
This implies $\lim_{r \to 0^+} r^{-\eta} {W}(r) = \infty$, since $\eta > 2-\alpha^{-1}+2\varepsilon$ by our choice of $\varepsilon>0$. Hence, by \eqref{Eq:KX08} we obtain
$\dimp G_X([0\,,1])\leq2-\alpha^{-1}$ almost surely, since $\eta>2-\alpha^{-1}$ is arbitrary, concluding the proof of Proposition \ref{dimgraph1}.
\end{proof}

\begin{proof}[Proof of Theorem \ref{dimgraph}]
We will first prove the lower bounds. In case $\alpha_1\leq d_1$ clearly $\dim_{\rm H}G_X([0\,,1])\geq1$ by projecting the graph $\{(t,X(t))\}_{t\geq0}$ onto the first (deterministic) component and by projection of the graph onto the range we get $\dim_{\rm H}G_X([0\,,1])\geq\alpha_1$ almost surely by \eqref{dimrange}. In case $\alpha_1>d_1$ we have $d_1=1$, hence by projecting the graph $\{(t,X(t))\}_{t\geq0}$ onto the subgraph $\{(t,X^{(1)}(t))\}_{t\geq0}$ we get $\dim_{\rm H}G_X([0\,,1])\geq\dim_{\rm H}G_{X^{(1)}}([0\,,1])=2-\alpha_1^{-1}$ almost surely by Proposition \ref{dimgraph1} and a projection of the graph onto the range yields $\dim_{\rm H}G_X([0\,,1])\geq 1+\alpha_2(1+\alpha_1^{-1})$ almost surely by \eqref{dimrange}.

Since $\dim_{\rm H}G_X([0\,,1])\leq\dim_{\rm P}G_X([0\,,1])$, again it remains to prove the upper bound for the packing dimension. To prove this upper bound, we rewrite the tail indices $\alpha_0=1$ and $\alpha_1>\cdots>\alpha_p$ for each of the $d+1$ coordinates so that $\tilde\alpha_0\geq\cdots\geq\tilde\alpha_d$. In principle, we now have to distinguish four cases:
\begin{itemize}
\item[1.)] $\alpha_1\leq1=d_1$, then we have $\tilde\alpha_0=\alpha_0=1$, $\tilde\alpha_1=\alpha_1$ and we need to show that $r^{-\eta}W(r)\to\infty$ as $r\to0^+$ for all $\eta>1$.
\item[2.)] $1<\alpha_1<2\leq d_1$, then $\tilde\alpha_0=\alpha_1=\tilde\alpha_1$ and we have to show that $r^{-\eta}W(r)\to\infty$ as $r\to0^+$ for all $\eta>\alpha_1$.
\item[3.)] $\alpha_1>1=d_1\geq\alpha_2$, then $\tilde\alpha_0=\alpha_1$, $\tilde\alpha_1=\alpha_0=1$ and we have to show that $r^{-\eta}W(r)\to\infty$ as $r\to0^+$ for all $\eta>2-\alpha_1^{-1}$.
\item[4.)] $\alpha_1>\alpha_2>1=d_1$, then $\tilde\alpha_0=\alpha_1$, $\tilde\alpha_1=\alpha_2$ and we have to show that $r^{-\eta}W(r)\to\infty$ as $r\to0^+$ for all $\eta>1+\alpha_2(1-\alpha_1^{-1})$.
\end{itemize}
Note that these four cases can be summarized in the sense that we have to show $r^{-\eta}W(r)\to\infty$ as $r\to0^+$ 
for all $\eta>1+\tilde\alpha_1(1-\tilde\alpha_0^{-1})\geq1$, where $2\geq\tilde\alpha_0\geq\tilde\alpha_1$ and $\tilde\alpha_0\geq1$. 
We write $\tilde\xi=(\tilde\xi_0,\ldots,\tilde\xi_d)\in\rr^{d+1}$ and define
$$i^\ast:=\min\{i=0,\ldots,d:\,\tilde\alpha_i=1=\alpha_0\}.$$
Then it follows from \eqref{def:W}, \eqref{intnew} and Theorem \ref{main} that for $r\in(0,1)$ we have
\begin{align*}
W(r) & =\int_{\rr^{d+1}} \Re\left( \frac{1}{1+\tilde\psi(\tilde\xi/r)}\right)\frac{d\tilde\xi}{\prod_{i=0}^d(1+ \tilde\xi_i^2)}\\
& \geq K\int_{\{|\tilde\xi_i|\geq\tau r,\,0\leq i\leq d\}}\frac{\sum_{i\not= i^\ast}|\tilde\xi_i/r|^{\tilde\alpha_i}}
{\big(\|\tilde\xi/r\|^{\varepsilon/2}\displaystyle\sum_{i\not= i^\ast}|\tilde\xi_i/r|^{\tilde\alpha_i}\big)^2+\big(\|\tilde\xi/r\|^{\varepsilon/2}\displaystyle\sum_{i\not= i^\ast}|\tilde\xi_i/r|^{\tilde\alpha_i}+|\tilde\xi_{i^\ast}/r|\big)^2}\,\frac{d\tilde\xi}{\prod_{i=0}^d(1+\tilde\xi_i^2)}\\
& \geq K\int_{\left\{{\tilde\xi_i\geq\tau r,\,0\leq i\leq d\atop \tilde\xi_i\leq1,\,2\leq i\leq d}\right\}}\frac{\sum_{i\not= i^\ast}(\tilde\xi_i/r)^{\tilde\alpha_i}}{\big(\sum_{i= 0}^d(\tilde\xi_i/r)^{\tilde\alpha_i}\big)^2}\,\frac{d\tilde\xi}{\|\tilde\xi/r\|^{\varepsilon}(1+\tilde\xi_0^2)(1+\tilde\xi_1^2)}\\
& \geq K\int_{\{\tau r\leq\tilde\xi_i\leq1,\,2\leq i\leq d\}}\,d\tilde\xi_2\cdots d\tilde\xi_d\\
& \qquad\cdot\int_{\tau r}^\infty\int_{\tau r}^\infty
\frac{C+\sum_{i\in\{0,1\}\setminus\{ i^\ast\}}(\tilde\xi_i/r)^{\tilde\alpha_i}}{(A+(\tilde\xi_1/r)^{\tilde\alpha_1}+(\tilde\xi_0/r)^{\tilde\alpha_0})^2(B+(\tilde\xi_1/r)^{2}+(\tilde\xi_0/r)^{2})^{\varepsilon/2}}\,\frac{d\tilde\xi_0}{1+\tilde\xi_0^2}\,\frac{d\tilde\xi_1}{1+\tilde\xi_1^2},
\end{align*}
where
$$A=\sum_{i=2}^d(\tilde\xi_i/r)^{\tilde\alpha_i},\qquad B=\sum_{i=2}^d(\tilde\xi_i/r)^{2},\qquad C=\sum_{i\in\{2,\ldots,d\}\setminus\{i^\ast\}}(\tilde\xi_i/r)^{\tilde\alpha_i}.$$
For fixed $(\tilde\xi_2,\ldots,\tilde\xi_d)\in[\tau r,1]^{d-1}$ and thus fixed $A,B,C$ we consider the inner integrals
$$I(r):=\int_{\tau r}^\infty\int_{\tau r}^\infty
\frac{C+\sum_{i\in\{0,1\}\setminus\{ i^\ast\}}(\tilde\xi_i/r)^{\tilde\alpha_i}}{(A+(\tilde\xi_1/r)^{\tilde\alpha_1}+(\tilde\xi_0/r)^{\tilde\alpha_0})^2(B+(\tilde\xi_1/r)^{2}+(\tilde\xi_0/r)^{2})^{\varepsilon/2}}\,\frac{d\tilde\xi_0}{1+\tilde\xi_0^2}\,\frac{d\tilde\xi_1}{1+\tilde\xi_1^2}.$$

In case $i^\ast=0$, i.e.\ $1=\tilde\alpha_0\geq\tilde\alpha_1$, let $\eta>1$ be arbitrary and choose $\varepsilon>0$ such that $\eta>1+2\varepsilon$. Then we have
\begin{align*}
I(r) & =\int_{\tau r}^\infty\int_{\tau r}^\infty
\frac{C+(\tilde\xi_1/r)^{\tilde\alpha_1}}{(A+(\tilde\xi_1/r)^{\tilde\alpha_1}+(\tilde\xi_0/r)^{\tilde\alpha_0})^2}\,\frac{d\tilde\xi_1}{(B+(\tilde\xi_1/r)^{2}+(\tilde\xi_0/r)^{2})^{\varepsilon/2}(1+\tilde\xi_1^2)}\,\frac{d\tilde\xi_0}{1+\tilde\xi_0^2}\\
& \geq K\, r^{\tilde\alpha_1+\varepsilon}\int_{\tau r}^\infty\int_{\tau}^\infty
\frac{\tilde\xi_1^{\tilde\alpha_1}}{(r^{\tilde\alpha_1}A+\tilde\xi_1^{\tilde\alpha_1}+r^{\tilde\alpha_1-\tilde\alpha_0}\tilde\xi_0^{\tilde\alpha_0})^2}\,\frac{d\tilde\xi_1}{(r^2B+\tilde\xi_1^{2}+\tilde\xi_0^{2})^{\varepsilon/2}(1+\tilde\xi_1^2)}\,\frac{d\tilde\xi_0}{1+\tilde\xi_0^2}\\
& \geq K\, r^{\tilde\alpha_1+\varepsilon}\int_{\tau r^{1-\tilde\alpha_1-\varepsilon}}^1
\frac{1}{(1+r^{\tilde\alpha_1-\tilde\alpha_0}\tilde\xi_0^{\tilde\alpha_0})^2}\int_{\tau}^\infty\frac{d\tilde\xi_1}{(r^2B+\tilde\xi_1^{2}+1)^{\varepsilon/2}\tilde\xi_1^{\tilde\alpha_1}(1+\tilde\xi_1^2)}\,\frac{d\tilde\xi_0}{1+\tilde\xi_0^2},
\end{align*}
where in the last step we used $r^{\tilde\alpha_1}A\leq d-1$, $\tilde\xi_1^{-\tilde\alpha_1}\leq\tau^{-\tilde\alpha_1}$, $r^2B\leq d-1$ and $\tilde\xi_0\leq 1$. Since the inner integral is positive and finite, $\tilde\alpha_0=1$ and for $\tilde\xi_0\geq \tau r^{1-\tilde\alpha_1-\varepsilon}$ we have $r^{\tilde\alpha_1-\tilde\alpha_0}\tilde\xi_0^{\tilde\alpha_0}\geq K\, r^{-\varepsilon}\to\infty$ as $r\to0^+$, for sufficiently small $r\in(0,1)$ we further get
\begin{align*}
I(r) & \geq K\, r^{\tilde\alpha_1+\varepsilon}\int_{\tau r^{1-\tilde\alpha_1-\varepsilon}}^1
\frac{d\tilde\xi_0}{r^{2\tilde\alpha_1-2\tilde\alpha_0}\tilde\xi_0^{2\tilde\alpha_0}(1+\tilde\xi_0^2)}\\
& \geq K\, r^{2-\tilde\alpha_1+\varepsilon}\int_{\tau r^{1-\tilde\alpha_1-\varepsilon}}^1
\frac{d\tilde\xi_0}{\tilde\xi_0^{2}}\geq K\,r^{1+2\varepsilon}.
\end{align*}
This implies $\lim_{r \to 0^+} r^{-\eta} {W}(r) = \infty$, since $\eta>1+2\varepsilon$ by our choice of $\varepsilon>0$. Hence, by \eqref{Eq:KX08} we obtain $\dimp G_X([0\,,1])\leq1=1+\tilde\alpha_1(1-\tilde\alpha_0^{-1})$ almost surely, since $\eta>1$ is arbitrary.

In case $i^\ast\geq1$, i.e.\ $\tilde\alpha_0>1\geq\tilde\alpha_1$, let $\eta>1+\tilde\alpha_1(1-\tilde\alpha_0^{-1})$ be arbitrary and choose $\varepsilon>0$ such that $\eta>1+\tilde\alpha_1(1-\tilde\alpha_0^{-1})+2\varepsilon$. Note that
$$\beta:=\frac{1+\tilde\alpha_1(1-\tilde\alpha_0^{-1})-\tilde\alpha_0+\varepsilon}{1-2\tilde\alpha_0}=1+\frac{\tilde\alpha_1(1-\tilde\alpha_0^{-1})+\varepsilon}{1-2\tilde\alpha_0}<1$$
and observe that $\tilde\alpha_1-\tilde\alpha_0(1-\beta)=\varepsilon\tilde\alpha_0/(1-2\tilde\alpha_0)<0$ by elementary calculations. Then we have
\begin{align*}
I(r) & \geq\int_{\tau r}^\infty\int_{\tau r}^\infty
\frac{C+(\tilde\xi_0/r)^{\tilde\alpha_0}}{(A+(\tilde\xi_1/r)^{\tilde\alpha_1}+(\tilde\xi_0/r)^{\tilde\alpha_0})^2}\,\frac{d\tilde\xi_1}{(B+(\tilde\xi_1/r)^{2}+(\tilde\xi_0/r)^{2})^{\varepsilon/2}(1+\tilde\xi_1^2)}\,\frac{d\tilde\xi_0}{1+\tilde\xi_0^2}\\
& \geq K\, r^{2\tilde\alpha_1-\tilde\alpha_0+\varepsilon}\int_{\tau r}^\infty\int_{\tau}^\infty
\frac{\tilde\xi_0^{\tilde\alpha_0}}{(r^{\tilde\alpha_1}A+\tilde\xi_1^{\tilde\alpha_1}+r^{\tilde\alpha_1-\tilde\alpha_0}\tilde\xi_0^{\tilde\alpha_0})^2}\,\frac{d\tilde\xi_1}{(r^2B+\tilde\xi_1^{2}+\tilde\xi_0^{2})^{\varepsilon/2}(1+\tilde\xi_1^2)}\,\frac{d\tilde\xi_0}{1+\tilde\xi_0^2}\\
& \geq K\, r^{2\tilde\alpha_1-\tilde\alpha_0+\varepsilon}\int_{\tau r^\beta}^1
\frac{\tilde\xi_0^{\tilde\alpha_0}}{(1+r^{\tilde\alpha_1-\tilde\alpha_0}\tilde\xi_0^{\tilde\alpha_0})^2}\int_{\tau}^\infty\frac{d\tilde\xi_1}{(r^2B+\tilde\xi_1^{2}+1)^{\varepsilon/2}\tilde\xi_1^{2\tilde\alpha_1}(1+\tilde\xi_1^2)}\,\frac{d\tilde\xi_0}{1+\tilde\xi_0^2},
\end{align*}
where in the last step we used $r^{\tilde\alpha_1}A\leq d-1$, $\tilde\xi_1^{-\tilde\alpha_1}\leq\tau^{-\tilde\alpha_1}$, $r^2B\leq d-1$ and $\tilde\xi_0\leq 1$. Since the inner integral is positive and finite and for $\tilde\xi_0\geq\tau r^\beta$ we have $r^{\tilde\alpha_1-\tilde\alpha_0}\tilde\xi_0^{\tilde\alpha_0}\geq K\,r^{\tilde\alpha_1-\tilde\alpha_0(1-\beta)}\to\infty$ as $r\to0^+$ by our choice of $\beta$, for sufficiently small $r\in(0,1)$ we further get
\begin{align*}
I(r) & \geq K\, r^{2\tilde\alpha_1-\tilde\alpha_0+\varepsilon}\int_{\tau r^{\beta}}^1
\frac{d\tilde\xi_0}{r^{2\tilde\alpha_1-2\tilde\alpha_0}\tilde\xi_0^{2\tilde\alpha_0}(1+\tilde\xi_0^2)}\\
& \geq K\, r^{\tilde\alpha_0+\varepsilon}\int_{\tau r^{\beta}}^1
\frac{d\tilde\xi_0}{\tilde\xi_0^{2\tilde\alpha_0}}\geq K\,r^{\tilde\alpha_0+\varepsilon+\beta(1-2\tilde\alpha_0)}=K\,r^{1+\tilde\alpha_1(1-\tilde\alpha_0^{-1})+2\varepsilon}.
\end{align*}
This implies $\lim_{r \to 0^+} r^{-\eta} {W}(r) = \infty$, since $\eta>1+\tilde\alpha_1(1-\tilde\alpha_0^{-1})+2\varepsilon$ by our choice of $\varepsilon>0$. Hence, by \eqref{Eq:KX08} we obtain $\dimp G_X([0\,,1])\leq 1+\tilde\alpha_1(1-\tilde\alpha_0^{-1})$ almost surely, since $\eta>1+\tilde\alpha_1(1-\tilde\alpha_0^{-1})$ is arbitrary, concluding the proof.
\end{proof}

\subsection{Recurrence and Transience}

A L\'evy process $X=\{X(t)\}_{t\geq0}$ in $\rd$ is called {\it recurrent} if $\liminf_{t\to\infty}\|X(t)\|=0$ almost surely and it is called {\it transient} if $\lim_{t\to\infty}\|X(t)\|=\infty$. Due to dichotomy, every L\'evy process is either recurrent or transient; e.g., see Theorem 35.4 in \cite{Sat}. In case of a full, strictly operator semistable L\'evy process recurrence and transience of $X$ is fully characterized by the following results.
\begin{itemize}
\item $d\geq3$: Every full L\'evy process in $\rd$ is transient by Theorem 37.8 in \cite{Sat}.
\item $d=2$: By Choi and Sato \cite{ChoSat} the only recurrent strictly operator-semistable L\'evy processes in $\rr^2$ are Gaussian with $\alpha_1=2$ and $d_1=2$; see also Theorem 37.18 in \cite{Sat}.
\item $d=1$: By Theorems 3.1 and 3.2 in Choi \cite{Choi} a strictly $\alpha$-semistable L\'evy process in $\rr$ is recurrent if $1\leq\alpha\leq2$ and it is transient if $0<\alpha<1$.
\end{itemize}
Hence, together with \eqref{dimr2} and \eqref{dimr1} we immediately get a characterization of recurrence and transience by the Hausdorff dimension of the range.

\begin{cor}\label{recdim}
A full, strictly operator semistable L\'evy process $X=\{X(t)\}_{t\geq0}$ in $\rd$ is recurrent if and only if $\dim_{\rm H}X([0,1])=d$ almost surely.
\end{cor}

A possible interpretation of this result is that a strictly operator semistable L\'evy process is recurrent if and only if its sample paths are almost surely ``space filling''. Note that Corollary \ref{recdim} is not true for arbitrary L\'evy processes as follows. By Theorem 37.5 in \cite{Sat} a L\'evy process $X$ in $\rd$ is recurrent if and only if
\begin{equation}\label{reccrit}
\lim_{q\downarrow0}\int_{\{\|\xi\|<1\}}\Re\left(\frac1{q-\psi(\xi)}\right)\,d\xi=\infty.
\end{equation}
Hence recurrence and transience are determined by the behavior of $\psi(\xi)$ near the origin $\xi=0$, i.e.\ the tail behavior of the process, whereas by \eqref{Hdimcrit} the Hausdorff dimension of the range is determined by the behavior of $\psi(\xi)$ as $\|\xi\|\to\infty$, i.e.\ the local behavior of the process. To illustrate this we give the following example.

\begin{example}
Let $\phi$ be the symmetric L\'evy measure on $\rr^\ast=\rr\setminus\{0\}$ with Lebesgue density
$$g(x)=\begin{cases}|x|^{-(\beta+1)} & \text{ if }0<|x|\leq1,\\
|x|^{-(\alpha+1)} & \text{ if }|x|>1,\end{cases}$$
where $\alpha>0$ and $\beta<2$ due to $\int_{\rr^\ast}(1\wedge x^2)\,d\phi(x)<\infty$. Then it can be easily shown using the criteria \eqref{reccrit} and \eqref{Hdimcrit} that the L\'evy process $X=\{X(t)\}_{t\geq0}$ with L\'evy exponent
$$\psi(\xi)=\int_{\rr^\ast}\left(e^{i\xi x}-1-\frac{i\xi x}{1+x^2}\right)\,d\phi(x)=\int_{\rr^\ast}\left(\cos(\xi x)-1\right)g(x)\,dx$$
is recurrent if and only if $\alpha\geq1$ and we have $\dim_{\rm H}X([0,1])=\max\{0,\min\{\beta,1\}\}$ almost surely. 
For $\beta<1\leq\alpha$ or $\alpha<1\leq\beta$ we see that the statement of Corollary \ref{recdim} fails to hold.
\end{example}

\begin{remark}
It is also possible to characterize recurrence and transience of a full, strictly operator semistable L\'evy process $X$ 
by the Hausdorff dimension of its graph, but we have to distinguish between the cases $d=1$ and $d\geq2$. 
A comparison of the above results of Choi and Sato with \eqref{dimg2} and \eqref{dimg1} easily gives that a full, 
strictly operator semistable L\'evy process is recurrent if and only if
$$\dim_{\rm H}G_X([0,1])=\begin{cases}2-\alpha^{-1 } & \text{ if }d=1\\ d  & \text{ if }d\geq2\end{cases}$$
almost surely.
\end{remark}

\subsection{Double points of operator semistable L\'evy processes}

A point $x\in\rd$ is called a $k$-{\it multiple point} of the stochastic process $X$ for some $k\in\nat\setminus\{1\}$ if there exist $0\leq t_1<\cdots<t_k$ such that $X(t_1)=\cdots=X(t_k)=x$. By $M_X(k)$ we denote the set of all $k$-multiple points of $X$. Recently, Luks and Xiao \cite{LX} derived a general formula for the Hausdorff dimension of $M_X(k)$ for symmetric, absolutely continuous L\'evy processes in terms of the L\'evy exponent $\psi$. In Theorem 1 of \cite{LX} they proved that
$$\dim_{\rm H}M_X(k)=d-\inf\bigg\{\beta\in(0,d]:\,\int_{\rr^{kd}}\frac{1}{1+\left\|\sum_{j=1}^k\xi^{(j)}\right\|^\beta}\prod_{j=1}^k\frac1{1+\psi(\xi^{(j)})}\,d\bar\xi<\infty\bigg\},$$
almost surely with the convention $\inf\emptyset=d$, where $\bar\xi=(\xi^{(1)},\ldots,\xi^{(d)})$ for $\xi^{(j)}\in\rd$. Moreover, in case of symmetric operator stable L\'evy processes $X$ with exponent $E$, Luks and Xiao \cite{LX} were able to calculate $M_X(2)$ explicitly, based only on the fact that for $\varepsilon>0$ there exists $\tau>1$ such that for some $K>1$ it holds that
$$\frac{K^{-1}\|\xi\|^{-\varepsilon}}{\sum_{i=1}^p\|\xi_i\|^{\alpha_i}}\leq\frac1{1+\psi(\xi)}\leq\frac{K\|\xi\|^{\varepsilon}}{\sum_{i=1}^p\|\xi_i\|^{\alpha_i}}\quad\text{ for all }\|\xi\|>\tau,$$
which is known from (4.2) of \cite{MX}. Since from Corollary \ref{proppsi} we know that the same bounds hold true also for symmetric operator semistable L\'evy processes, the explicit dimension results of \cite{LX} also hold in this more general situation. In the following we reformulate Theorem 2 of \cite{LX} for the semistable case, where we rearrange the distinct real parts $\alpha_1>\cdots>\alpha_p$ of the eigenvalues of the exponent $E$ as $\tilde\alpha_1\geq\cdots\geq\tilde\alpha_d$ including their multiplicities.

\begin{cor}\label{Co_M2}
Let $X=\{X(t)\}_{t\geq0}$ be a symmetric operator semistable L\'evy process in $\rd$ with exponent $E$. Then for the double points of $X$ we have almost surely:
\begin{itemize}
\item[(a)] If $d=2$ then $\quad\dim_{\rm H}M_X(2)=\min\left\{\tilde\alpha_1(2-\tilde\alpha_1^{-1}-\tilde\alpha_2^{-1}),2\tilde\alpha_2(1-\tilde\alpha_1^{-1})\right\}$.
\item[(b)] If $d=3$ then $\quad\dim_{\rm H}M_X(2)=\tilde\alpha_1(2-\tilde\alpha_1^{-1}-\tilde\alpha_2^{-1}-\tilde\alpha_3^{-1})$.
\item[(c)] If $d\geq4$ then $\quad M_X(2)=\emptyset$.
\end{itemize}
Note that a negative Hausdorff dimension means that $M_X(2)=\emptyset$ almost surely.
\end{cor}

\subsection{Concluding remarks}
The results in Section 3 show that many sample path properties of a strictly operator semistable L\'evy processes can be described by the real parts of the eigenvalues of the exponent $E$, and the upper and lower bounds in Theorem \ref{main} play an important role in studying these and other properties. Several interesting questions remain open. For example, Corollary \ref{Co_M2}, as well as \cite{LX}, only provides information for the set of double points, it would be interesting to solve the problems for $k=3$ (for $k \ge 4$, the set of multiple points is almost surely empty). Moreover, Khoshnevisan and Xiao \cite{KX09}, Khoshnevisan Shieh and Xiao \cite{KSX} have studied the existence of intersections of independent L\'evy processes and the Hausdorff dimensions of the sets of intersection times and intersection points, respectively. Their results are expressed in terms of the L\'evy exponents of the processes. For strictly operator semistable L\'evy processes, we believe that these results could be explicitly expressed in terms of the real parts of the eigenvalues of their exponents.

\bibliographystyle{plain}

\end{document}